\documentclass[a4paper,12pt]{article}

\usepackage[T2A]{fontenc}
\usepackage{amsthm,amssymb}
\usepackage{inputenc}
\usepackage[english]{babel}
\usepackage{amsfonts}
\usepackage{amsmath}

\newtheorem{thm}{Theorem}
\newtheorem{lem}{Lemma}

\newtheorem{cor}{Corollary}

\begin{document}

\title{Hardy type inequalities and parametric Lamb equation}

\author{Makarov R.V., Nasibullin R.G.}

\maketitle
\date{}

\begin{abstract}
This paper is devoted to Hardy type inequalities with remainders for compactly supported smooth functions on open sets in the Euclidean space. We establish new inequalities with weight functions depending on the distance function to the boundary of the domain. One-dimensional $L_1$ and $L_p$ inequalities and their multidimensional analogues are proved. We consider spatial inequalities  in open convex  domains with the finite inner radius. Constants in these inequalities depend on the roots of parametric Lamb equation for the Bessel function and turn out to be sharp in some particular cases.
\end{abstract}

\textbf{subclass: }{26D10; 26D15} 

\textbf{keywords:} {Hardy type inequality, distance to a boundary, finite inner radius,  additional term, Bessel function} 

\section{Introduction}

Let $\Omega$  be an open proper subset of the Euclidean space $\mathbb{R}^n$. Denote by $C_0^1 (\Omega)$ the family of continuously differentiable functions $f:\Omega \to\mathbb{R}$ with compact supports lying in $\Omega$.

In this study we consider inequalities of Hardy type with weight functions depending on the distance function $\delta(x)$ to the boundary of the domain  $\Omega$, i.e.
$$
\delta(x)=\delta(x,\Omega):=\textrm{dist}(x,\partial\Omega).
$$
Hardy-type inequalities related to the distance function have been studied for a long time (see for example \cite{AvLap,Avkh_IM,AW_ZAMM,AW_Lamb,Avkh_Nas_SibMJ_2014,Avkh_AIA,Avkh_JMA,Av3,Av4,BEL,BM} and references therein). Let us remark that Hardy type inequalities with weight functions depending on the hyperbolic radius (\cite{APR,AvkhMS_15,ANSh_2018,ANSh_2019,F,FR}), distance to $\mathbb{R}_{+}$ \cite{Maz,Tidblom1}, distance to the origin \cite{BEL} are well known also.

For instance, the following $n$-dimensional Hardy inequality for convex domains
$$
\int\limits_\Omega|\nabla f(x)|^2dx\geq \frac{1}{4}\int\limits_\Omega\frac{|f(x)|^2}{\delta(x)^2}dx
$$
is valid for any $f\in C_0^1 (\Omega)$. It is well known that the constant $1/4$ is sharp for any convex subdomain of $\mathbb{R}^n$ although there is no function $f\not\equiv 0$ for which equality in this inequality is actually attained (see \cite{BM,BST,Bur,Bur1,Dav,Dav1,MatSob,MMP}). We refer to  \cite{AvLap,Avkh_IM,MMP}  for related results in non-convex domains.

More precisely, this paper is devoted to generalization and extension of the following sharp inequality proved by F.G. Avkhadiev  and K.-J. Wirths in \cite{AW_ZAMM}:
\emph{
If $m>0$, $0<\nu\leq 1/m$ and $\Omega$ is a convex domain with finite inner radius
$$
\delta_0 (\Omega)=\sup\{\delta(x):x\in\Omega\},
$$
then for any $f\in C_0^1(\Omega)$ the following sharp inequality holds
\begin{equation}\label{intr_f1}
\frac{1-\nu^2m^2}{4}\int\limits_\Omega\frac{|f(x)|^2}{\delta(x)^2}dx+\frac{C^2}{\delta_0(\Omega)^m}\int\limits_\Omega\frac{|f(x)|^2}{\delta(x)^{2-m}}dx \leq \int\limits_\Omega|\nabla f(x)|^2dx,
\end{equation}
where $C= C_\nu(m)$ is the constant satisfying the equation
\begin{equation}\label{intr_lamb_eq}
J_\nu\left(\frac{2}{m} C\right)+2CJ_\nu' \left(\frac{2}{m} C\right)=0
\end{equation}		
for  the Bessel function $J_\nu$ of order  $\nu$.}

Following papers \cite{AW_ZAMM} and \cite{AW_Lamb}, we call the quantity $C_\nu(m)$ defined as a positive root of the equation (\ref{intr_lamb_eq}) the Lamb constant (see also \cite{Nas_Ufa_2017} and \cite{Nas_LJM_2016}). In addition, equations of the form (\ref{intr_lamb_eq}) we call Lamb equation.

Note that inequality (\ref{intr_f1}) is valid also for functions that belong the closure $H_0^1(\Omega)$ of the family $C_0^1 (\Omega)$ of smooth function with finite Dirichlet integral and supported in~ $\Omega$. It should be also stressed  that multidimensional inequality (\ref{intr_f1})  is a bridge between Hardy's and Poincare's inequalities and has some specialties like an additional term and  sharp constants (see \cite{AW_ZAMM} for more information). That is why we believe that this inequality is one of the most beautiful inequalities.

 M.~Hoffmann-Ostenhof, T.~Hoffmann-Ostenhof and A.~Laptev  \cite{HoHoL} proved that for all $f\in H_0^1(\Omega)$ the following Hardy type inequality
\begin{equation}\label{in_HHL}
\int\limits_\Omega|\nabla f(x)|^2dx\geq \frac{1}{4}\int\limits_\Omega \frac{|f|^2}{\delta(x)^{2}}dx+\frac{1}{4}\frac{K(n)}{|\Omega|^{2/n}}\int\limits_\Omega|f(x)|^2dx
\end{equation}
holds, where $\Omega$ is a convex domain in $\mathbb{R}^{n} \;(n \geq 2)$, $|\mathbb{S}^{n-1}|$  is the surface area of the unit sphere $\mathbb{S}^{n-1}$ in the Euclidean space $\mathbb{R}^n$, $|\Omega|$  is the volume of the set $\Omega$ and
$$
K(n) = n\left(\frac{|\mathbb{S}^{n-1}|}{n}\right)^{2/n}.
$$
The the constant $1/4$ is optimal. In \cite{EL}, W.D. Evans and Roger T. Lewis  improved the constant in the additional term of (\ref{in_HHL}). 

We refer to \cite{AW_Lamb,EL,HoHoL,Nas_Sib_2019,Nas_Ufa_2017,Nas_LJM_2019,Nas_LJM_2016,Nas_MS_2019,PS,Tidblom}  for other interesting Hardy type inequalities with additional nonnegative terms. We pay attention only to the fact that Hardy inequalities with remainders were first obtained by V.G. Maz'ya \cite{Maz} in the case where $\Omega$ is a half-space and renewed interest in Hardy type inequalities with additional non-negative terms followed the work of H.~Brezis and M.~Marcus \cite{BM}.


Hardy inequalities are used in many branches of mathematical analysis and mathematical physics. For instance, one dimensional Hardy type inequalities are regarded as a tool from the theory of functions employed in the proofs of embedding theorems for functional spaces (see the monographs by S.L.~Sobolev \cite{Sobolev} and V.G.~Maz'ya \cite{Maz}).

In \cite{Tidblom}, J. Tidblom  proved the following  inequality for functions $f$  from  the corresponding Sobolev  space:
\begin{multline}\label{in_intr}
c_p\left(\int\limits_\Omega \frac{|f(x)|^p}{\delta^p(x)}dx
+\frac{(p-1)\sqrt{\pi}\;
\Gamma(\frac{n+p}{2})}{\Gamma(\frac{p+1}{2})\;\Gamma(\frac{n}{2})}\Biggl(\frac{|\mathbb{S}^{n-1}|}{n|\Omega|}\Biggr)^{\frac{p}{n}}\int\limits_\Omega|f(x)|^p
dx\right) \leq \int\limits_\Omega|\nabla f(x)|^p dx,
\end{multline}
where $p>1$,  $\Omega$ is a convex domain in $\mathbb{R}^{n} \;(n \geq 2)$,   $|\mathbb{S}^{n-1}|$  is the surface area of the unit sphere $\mathbb{S}^{n-1}$ in the Euclidean space $\mathbb{R}^n$, $|\Omega|$  is the volume of the set $\Omega$, $\Gamma$  is the gamma  function of Euler and
$$
c_p=\left(\frac{p-1}{p}\right)^p.
$$
The constant $c_p$ is the best one. Notice that inequality (\ref{in_intr}) is an extension of the  inequality (\ref{in_HHL})  for $p=2$ in the case of arbitrary $p>1$.

As we mentioned above we get generalization and extension of (\ref{intr_f1}). Namely, we prove $L_p$ version of the inequality  for $p\geq 1$. It is well known that for $L_p$-spaces with $0 < p < 1$ the Hardy inequality is not satisfied for arbitrary non-negative measurable functions. In spite of this Hardy type inequalities for $0< p < 1$ also known for non-negative non-increasing functions (see \cite{BST, Bur, Bur1}).

Let the constant  $c_\nu(m)$  be a number satisfying the conditions
$$
(1-2\lambda)J_\nu\left(\frac{2}{m}c_\nu(m)\right)+2c_\nu(m)J'_\nu\left(\frac{2}{m}c_\nu(m)\right)= 0\quad\text{and}\quad \frac{2}{m}c_\nu(m)\in(0,j_{\nu}),
$$
where $j_\nu$ is the first positive zero of the Bessel function $J_\nu$ of order  $\nu$. In this paper we assume that   $c_\nu(m)$ is the first root of the equation.

The last equation  for the Bessel function we will call parametric Lamb equation.   It is clear, if $\lambda = 0$, then the last equation coincides with equation (\ref{intr_lamb_eq}) that F.G.~Avkhadiev  and K.-J.~ Wirths used in \cite{AW_ZAMM}. We prove inequalities with constants depending on the roots of parametric Lamb equation for the Bessel function.

Let us consider 2 particular cases of this conditions.

\textbf{Example 1.} If $z=  \frac{2}{m} c_\nu(m)$, $\nu > 0$, $m>0$ and $\lambda = (1-m\nu)/2$, then this conditions take the form
$$
\nu J_\nu(z)+  z J'_\nu(z) = 0 \quad\text{and}\quad z\in(0,j_{\nu}).
$$
Since the following identity for the Bessel function holds
$$
\nu J_\nu(z)+  z J'_\nu(z) = z J_{\nu-1}(z),\ \ \  z>0, \ \ \ \nu >0,
$$
we get that  the Lamb constant $c_\nu (m) = \frac{m}{2}z =   \frac{m}{2}j_{\nu-1}$.

\textbf{Example 2.} If $\nu=1/2$ and $z=  \frac{2}{m}c_\nu(m)$, then this conditions  take the form
$$
2 m z \cos z - ( 4 \lambda + m -2) \sin z = 0 \quad\text{and}\quad  z\in (0,\pi).
$$
Note that in \cite{AW_Lamb}, F.G.~Avkhadiev and K.-J.~Wirths proved that the Lamb constant $z=\lambda_\nu(p)$ defined as the first positive root of the equation
$$
pJ_\nu (z) + 2zJ'_\nu(z)=0
$$  as a function in $p$ can be found as the solution of an initial value problem for the differential equation
$$
\frac{dz}{dp} = \frac{2z}{p^2-4\nu^2+4z^2}.
$$
For instance, the following are special cases of our results. Let  $\Omega$ be an open, convex set in $\mathbb{R}^n$ with finite inradius $\delta_0 = \delta_0(\Omega)$. If   $p\geq 1$, $m>1, \nu\in [0,1/m]$ and $\lambda\in\left[0,\frac{1+\nu m}{2}\right]$, then for any  $f\in C_0^1(\Omega)$ the following $L_p$-inequality
\begin{multline}\nonumber
(1-p\nu^2 m^2)\int\limits_\Omega \frac{|f(x)|^p}{\delta(x)^2} dx +\frac{4p\left(c_\nu^2(m)+(m-1)(\lambda -\lambda ^2)\right)}{\delta_0^m}\int\limits_\Omega \frac{|f(x)|^p}{\delta(x)^{2-m}} dx \leq\\
p^p\left(2(1+\nu m)-4\lambda^2\right)^p \int\limits_\Omega \frac{|\nabla f(x)|^{p}}{\delta(x)^{2-p}} dx
\end{multline}
is valid and if  $m>1, \nu\in [0,1/m)$ and $\lambda\in\left[0,\frac{1+\nu m}{2}\right]$, then for any  $f\in C_0^1(\Omega)$
\begin{multline}\nonumber
\int\limits_\Omega \frac{|f(x)|^p}{\delta(x)^2} dx +\frac{4p\left(c_\nu^2(m)+(m-1)(\lambda -\lambda ^2)\right)}{(1-\nu^2 m^2)\delta_0^m}\int\limits_\Omega \frac{|f(x)|^p}{\delta(x)^{2-m}} dx \leq\\
p^p \left(\frac{2}{1-\nu m}-\frac{4\lambda^2}{1-\nu^2m^2}\right)^p\int\limits_\Omega \frac{|f'(x)|^{p}}{\delta(x)^{2-p}} dx.
\end{multline}
Also, if   $m>1, \nu\in [0,1/m]$ and $\lambda\in\left[0,\frac{1+\nu m}{2}\right]$, then there is $L_2$-inequality
\begin{multline}\nonumber
\frac{1-\nu^2 m^2}{4}\int\limits_\Omega \frac{|g(x)|^2}{\delta(x)^2} dx +\frac{c_\nu^2(m)+(m-1)(\lambda-\lambda^2)}{\delta_0^m}\int\limits_\Omega \frac{|g(x)|^2}{\delta(x)^{2-m}} dx \leq\\
\leq 4\left(\frac{1+\nu m}{2}-\lambda^2\right) \int\limits_\Omega|\nabla g(x)|^2 dx.
\end{multline}
Let us compare the last inequality and sharp inequality (\ref{intr_f1}).
Obviously, if  $m>1$, $\nu\in [0,1/m]$, $\lambda\in\left[0,\frac{1+\nu m}{2}\right]$ and $\frac{1+\nu m}{2}-\lambda^2\leq 1/4$, then the first constant $(1-\nu^2 m^2)/4$ is sharp and
$$
c_\nu^2(m)+(m-1)(\lambda-\lambda^2)\leq  C_\nu^2(m).
$$
The paper is organized as follows. In Section \ref{sec1},  we prove one dimensional $L_1$ and $L_p$ inequalities. Also we get the new properties for the Bessel functions. To obtain $L_p$ inequalities we use H\"{o}lder's inequality and inequalities of Opial type (see \cite{Shum}). In Section \ref{sec2}, we establish inequalities in convex domains with the finite inner radius. We prove $L_1$, $L_2$ and $L_p$ spatial inequalities. To obtain multidimensional inequalities we use Avkhadiev's method that allow from one dimensional inequalities get  their corresponding spatial analogues. We refer to \cite{Av3} and \cite{Av4} for more information on this method (see also \cite{BEL}). Note that we give examples of our general results and compare  with known sharp inequalities. In particular, our inequalities with sharp constants (see, for instance, Example 3 and Corollary \ref{cor5} bellow).

\section{One dimensional inequalities}\label{sec1}
This section is devoted to one dimensional inequalities. We use these one dimensional inequalities to obtain their corresponding multidimensional analogues. To begin before we introduce some notations and give auxiliary statements. Also in this part  new properties of the Bessel function are obtained.

Suppose that $m>0, c>0$ and $0\leq\nu\leq \frac{1}{m}$. We will need the function
$$
y = f_{\nu,m}(x)= \sqrt{x}J_\nu\left(\frac{2}{m}c x^{m/2}\right),
$$
where by $J_\nu$ we denote the Bessel function of order $\nu$
$$
J_\nu(x) = \sum_{k=0}^\infty\frac{(-1)^kx^{2k+\nu}}{2^{2k+\nu }k! \Gamma (k+1+\nu)}, \ \  \ x\in[0,1].
$$
It is known (\cite{AW_ZAMM,AW_Lamb}) that for any $c > 0$ the function $f_{\nu,m}$ is a solution of the equation
\begin{equation}\label{difur}
y''+\left(\frac{1-\nu^2m^2}{4x^2}+\frac{c^2}{x^{2-m}}\right)y = 0, x\in \mathbb{R},
\end{equation}
and $f_{\nu,m}(x)>0, x\in (0,1]$ and $f'_{\nu,m}(x)>0, x\in (0,1)$.

See \cite{Watson} for more information about the Bessel functions and their properties.

\subsection{Auxiliary statements}

In the sequel we will use the following lemma.
\begin{lem}\label{le0}  Let $\nu>0$, $z\in(0,j_\nu)$, and let $j_\nu$ be the first positive zero of the Bessel function $J_\nu$ of order $\nu$. Then
$$
1-\left(\frac{J_{\nu-1}(z)-J_{\nu+1}(z)}{J_{\nu-1}(z)+J_{\nu+1}(z)}\right)^2-\frac{z^2}{\nu^2} =
\frac{z^2}{\nu^2}\left(-1+\frac{J_{\nu+1}(z)J_{\nu-1}(z)}{J_{\nu}^2(z)}\right).
$$
\end{lem}
\begin{proof}
Using the following equation for the Bessel function (see \cite{Watson})
$$
J_{\nu+1}(z) = \frac{2z}{\nu+1}\left(J_{\nu}(z)+J_{\nu+2}(z)\right)\quad\text{and}\quad J_{\nu-1}(z)=\frac{2\nu}{z} J_{\nu}(z)-J_{\nu+1}(z),
$$
by straightforward computation we obtain
\begin{multline}\nonumber
1-\frac{(J_{\nu-1}(z)-J_{\nu+1}(z))^2}{(J_{\nu-1}(z)+J_{\nu+1}(z))^2}-\frac{z^2}{\nu^2} = \\ \frac{4J_{\nu+1}(z)}{J_{\nu-1}(z)+J_{\nu+1}(z)}-\left(\frac{2J_{\nu+1}(z)}{J_{\nu-1}(z)+J_{\nu+1}(z)}\right)^2 -\frac{z^2}{\nu^2}=
\end{multline}
\begin{multline}\nonumber
=\frac{z^2}{\nu(\nu+1)}\frac{J_{\nu}(z)+J_{\nu+2}(z)}{J_{\nu}(z)}-\frac{z^2}{\nu^2}\frac{J_{\nu+1}^2(z)}{J_{\nu}^2(z)} -\frac{z^2}{\nu^2}=\\
\frac{z^2}{\nu^2}\left(\frac{\nu}{\nu+1}\frac{J_{\nu}(z)+J_{\nu+2}(z)}{J_{\nu}(z)}-\frac{J_{\nu+1}^2(z)}{J_{\nu}^2(z)} -1\right)=
\end{multline}
\begin{multline}\nonumber
=\frac{z^2}{\nu^2}\left(\frac{2\nu}{z}\frac{J_{\nu+1}(z)}{J_{\nu}(z)}-\frac{J_{\nu+1}^2(z)}{J_{\nu}^2(z)} -1\right)=\\\frac{z^2}{\nu^2}\left(\frac{J_{\nu+1}(z)}{J_{\nu}^2(z)}\left(\frac{2\nu}{z}J_{\nu}(z)-J_{\nu+1}(z)\right)-1\right)=
\end{multline}
$$
=\frac{z^2}{\nu^2}\left(-1+\frac{J_{\nu+1}(z)J_{\nu-1}(z)}{J_{\nu}^2(z)}\right).
$$
\end{proof}
The following lemma holds.

\begin{lem} \label{l1} If  $m>0$, $\nu\geq 0$ and $y = f_{\nu,m}(x)= \sqrt{x}J_\nu\left(\frac{2}{m}c x^{m/2}\right)$, $x\in[0,1]$, where $J_\nu$ is the Bessel function of order $\nu$. Then the function $\frac{xy'(x)}{y(x)}$ is decreasing and
$$
\sup_{x\in[0,1]}\frac{xy'(x)}{y(x)} = \lim_{x\to 0}\frac{xy'(x)}{y(x)} = \frac{1+\nu m}{2},
$$
$$
\inf _{x\in[0,1]}\frac{xy'(x)}{y(x)}  = \frac{y'(1)}{y(1)}=\frac{1}{2}+c\frac{J'_\nu\left(\frac{2}{m}c\right)}{J_\nu\left(\frac{2}{m}c\right)}.
$$
\end{lem}
\begin{proof}  Let $\nu =0$ and $z = \frac{2}{m} cx^{m/2}$. Obviously,
$$
\frac{d}{dx}\left(\frac{xy'(x)}{y(x)}\right) = \frac{y'(x)}{y(x)}+x\left(\frac{y''(x)y(x)-y'^2(x)}{y^2(x)}\right).
$$
Using equation (\ref{difur}) and that  $y(x)=  \sqrt{x}J_0\left(\frac{2}{m}c x^{m/2}\right)$, by straightforward computation we get
$$
\frac{d}{dx}\left(\frac{xy'(x)}{y(x)}\right) = \frac{1}{2x} +cx^{\frac{m}{2}-1}\frac{J_0'(z)}{J_0(z)}-\frac{1}{4x}-\frac{c^2}{x^{1-m}}-x\left(\frac{1}{2x} +cx^{\frac{m}{2}-1}\frac{J_0'(z)}{J_0(z)}\right)^2=
$$
$$
=-c^2x^{m-1}\left(1+\left(\frac{J'_0(z)}{J_0(z)}\right)^2\right)\leq 0, \ \ \ \  z\in (0,j_0),
$$
where  $j_0$ is the first positive zero of $J_0$.

Hence the function $xy'(x)/y(x)$ is decreasing. Consequently,
$$
\sup_{x\in[0,1]}\frac{xy'(x)}{y(x)} = \lim_{x\to 0}\frac{xy'(x)}{y(x)}
$$
and
$$
\inf _{x\in[0,1]}\frac{xy'(x)}{y(x)}  = \frac{y'(1)}{y(1)}=\frac{1}{2}+c\frac{J'_0\left(\frac{2}{m}c\right)}{J_0\left(\frac{2}{m}c\right)}.
$$
Now we suppose that $\nu >0$ and $z = \frac{2}{m}c x^{m/2}, x\in (0,1)$. Notice that if the first positive zero of $J_\nu$ is denoted by $j_\nu$, then $2c/m\in(0,j_\nu)$ and $z\in(0,j_\nu)$.

We also shall show that
$$
A:=\frac{d}{dx}\left(\frac{xy'(x)}{y(x)}\right)\leq 0.
$$
Obviously,
$$
A= \frac{y'(x)}{y(x)}+x\left(\frac{y''(x)}{y(x)}-\left(\frac{y'(x)}{y(x)}\right)^2\right).
$$
Using equation (\ref{difur}), we get that
$$
A=\frac{y'(x)}{y(x)}-\left(\frac{1-\nu^2m^2}{4x}+\frac{c^2}{x^{1-m}}\right)-x\left(\frac{y'(x)}{y(x)}\right)^2.
$$
Evidently,
$$
\frac{y'(x)}{y(x)} = \frac{1}{2x}+cx^{\frac{m}{2}-1}\frac{J'_\nu(z)}{J_\nu(z)}.
$$
Since (see \cite[p. 45]{Watson})
$$
J_\nu'(z) = \frac{1}{2} \left(J_{\nu-1}(z)-J_{\nu+1}(z)\right)\quad\text{and}\quad J_\nu(z) = \frac{z}{2\nu} \left(J_{\nu-1}(z)+J_{\nu+1}(z)\right),
$$
we have
$$
\frac{y'(x)}{y(x)} =\frac{1}{x}\left(\frac{1}{2}+\frac{m\nu}{2}\frac{J_{\nu-1}(z)-J_{\nu+1}(z)}{J_{\nu-1}(z)+J_{\nu+1}(z)}\right).
$$
Consequently,
$$
A= \frac{m^2\nu^2}{4x}\left(1-\left(\frac{J_{\nu-1}(z)-J_{\nu+1}(z)}{J_{\nu-1}(z)+J_{\nu+1}(z)}\right)^2-\frac{z^2}{\nu^2}\right).
$$
Using Lemma \ref{le0} and that $z = \frac{2}{m} cx^{m/2}$, obviously we get
$$
A= c^2x^{m-1}\left(-1+\frac{J_{\nu+1}(z)J_{\nu-1}(z)}{J_{\nu}^2(z)}\right).
$$
It is known that (see \cite[p. 152]{Watson})
$$
\frac{z^2}{4}\left(J_{\nu-1}^2(z)-J_{\nu-2}(z)J_\nu(z)\right)=\sum_{n=0}^\infty(\nu+2n)J_{\nu+2n}^2(z).
$$
Thus
$$
A= -\frac{q^2\lambda^2_\nu}{z^2}\frac{x^{q-1}}{J_\nu^2(z)}\sum_{n=0}^\infty(\nu+2n+1)J_{\nu+2n+1}^2(z)\leq 0.
$$
It remains to check that
$$
 \lim_{x\to 0}\frac{xy'(x)}{y(x)} = \lim_{x\to 0}\left(\frac{1}{2}+cx^{\frac{m}{2}}\frac{J'_\nu(\frac{2}{m}c x^{m/2})}{J_\nu(\frac{2}{m}c x^{m/2})}\right)=\frac{1+\nu m}{2}.
$$
Using the expansion for the Bessel function it is easy to obtain that
\begin{multline}\nonumber
\frac{J'_\nu(\frac{2}{m}c x^{m/2})}{J_\nu(\frac{2}{m}c x^{m/2})} = \frac{m}{2c}\frac{\sum_{k=0}^\infty\frac{(-1)^k(2k+\nu)(c/m)^{2k+\nu}x^{\frac{m}{2}(2k+\nu-1)}}{k! \Gamma (k+1+\nu)}}{\sum_{k=0}^\infty\frac{(-1)^k(c/m)^{2k+\nu}x^{\frac{m}{2}(2k+\nu)}}{k! \Gamma (k+1+\nu)}}=\\
\frac{m\nu}{2cx^{m/2}}\frac{1+\sum_{k=1}^\infty\frac{(-1)^k(2k/\nu+1)(c/m)^{2k}\Gamma(1+\nu)x^{mk}}{k! \Gamma (k+1+\nu)}}{1+\sum_{k=1}^\infty\frac{(-1)^k(c/m)^{2k}\Gamma(1+\nu)x^{mk}}{k! \Gamma (k+1+\nu)}}=\frac{m\nu}{2cx^{m/2}}G(x),
\end{multline}
where $G(x)\to 1$ as $x\to 0$.

Hence,
$$
\lim_{x\to 0}\left(\frac{1}{2}+cx^{\frac{m}{2}}\frac{J'_\nu(\frac{2}{m}c x^{m/2})}{J_\nu(\frac{2}{m}c x^{m/2})}\right)=\lim_{x\to 0}\left(\frac{1}{2}+\frac{m\nu}{2}G(x)\right)=\frac{1+\nu m}{2}.
$$
This concludes the proof of Lemma \ref{l1}.
\end{proof}
\begin{cor} \label{cor1}
If $\lambda\to \frac{1+\nu m}{2}$, then the lamb constant $c_\nu(m)\to 0$.
\end{cor}

\subsection{$L_1$ inequalities}

This part devoted to $L_1$ inequalities of Hardy type. We will use these inequalities to obtain $L_p$ inequalities. Note that $L_{1}$ inequalities also are of independent interest. For example,  in \cite{PS}, G.~Psaradakis obtained  sharp homogeneous improvements to $L_1$ weighted Hardy inequalities involving distance from the boundary.

The following assertion holds.

\begin{lem}\label{l3} Suppose that $m>0$,  $\nu\in (0,1/m]$ and let $f$ be an absolutely continuous function in $[0, 1]$ such that $f(0) = 0$ and $f'/x\in L^1[0, 1]$. If $\lambda\in\left[0,\frac{1+\nu m}{2}\right)$, then
\begin{multline}\nonumber
\frac{1-\nu^2 m^2}{4}\int\limits_0^1 \frac{|f(x)|}{x^2} dx +c_{\nu}^2(m)\int\limits_0^1 \frac{|f(x)|}{x^{2-m}} dx \leq\\
 \left(\frac{1+\nu m}{2}-\lambda^2\right) \int\limits_0^1 \frac{|f'(x)|}{x} dx+\left(\lambda^2-\lambda\right)\int\limits_0^1|f'(x)| dx,
\end{multline}
and if $ \lambda\leq 0$, then
$$
\frac{1-\nu^2 m^2}{4}\int\limits_0^1 \frac{|f(x)|}{x^2} dx +c_{\nu}^2(m)\int\limits_0^1 \frac{|f(x)|}{x^{2-m}} dx \leq \int\limits_0^1 \frac{|f'(x)|}{x}\left(\frac{1+\nu m}{2}-\lambda x\right) dx,
$$
where   $c= c_{\nu}(m)$  is a number satisfying the conditions
$$
(1-2\lambda) J_\nu\left(\frac{2}{m}c\right)+2cJ'_\nu\left(\frac{2}{m}c\right) = 0 \quad\text{and}\quad\frac{2}{m}c\in(0,j_{\nu}).
$$
\end{lem}
\begin{proof} Using the inequality $|f(x)|\leq \int_0^x|f'(t)|dt$ and changing the order of integration in a multiple integral, we obtain
$$
\int\limits_0^1 \frac{|f(x)|}{x^2} \left( \frac{1-\nu^2 m^2}{4}+c^2_\nu(m)x^{m} \right) dx \leq \int\limits_0^1 |f'(t)| T(t) dt,
$$
where
$$
T(t)  = \int\limits_t^1\left(\frac{1-\nu^2m^2}{4x^2}+\frac{c^2_\nu(m)}{x^{2-m}}\right) dx.
$$
By  differential equation  (\ref{difur}) we get
$$
T(t)  = -\int\limits_t^1 \frac{y''(x)}{y(x)} dx= -\int\limits_t^1 \left(\frac{y'(x)}{y(x)}\right)'+ \left(\frac{y'(x)}{y(x)}\right)^2 dx,
$$
since
$$
\frac{d}{dt}\left(\frac{y'(x)}{y(x)}\right)=\frac{y''(x)}{y(x)}-\left(\frac{y'(x)}{y(x)}\right)^2.
$$
Let us consider two cases:

\textbf{Case 1:}  $xy'(x)/y(x)\geq0, x\in [0,1]$.

Since $xy'(x)/y(x)$ is decreasing, it is easy to see that
$$
\min_{x\in[0,1]}\left(x\frac{y'(x)}{y(x)}\right)^2 = \left(\frac{y'(1)}{y(1)}\right)^2\geq 0
$$
and
$$
T(t)  =  -\int\limits_t^1 \left(\frac{y'(x)}{y(x)}\right)'+ \left(\frac{y'(x)}{y(x)}\right)^2 dx \leq \frac{y'(t)}{y(t)}-\frac{y'(1)}{y(1)}-\left(\frac{y'(1)}{y(1)}\right)^2\int\limits_t^1\frac{dx}{x^2}=
$$
$$
=\frac{y'(t)}{y(t)}-\frac{y'(1)}{y(1)}+\left(\frac{y'(1)}{y(1)}\right)^2-\left(\frac{y'(1)}{y(1)}\right)^2\frac{1}{t}.
$$
By Lemma \ref{l1} so that
$$
T(t) \leq \frac{1}{t}\sup_{t\in[0,1]}\frac{ty'(t)}{y(t)}-\frac{y'(1)}{y(1)}+\left(\frac{y'(1)}{y(1)}\right)^2-\left(\frac{y'(1)}{y(1)}\right)^2\frac{1}{t}=
$$
$$
=\left(\frac{1+\nu m}{2}-\left(\frac{y'(1)}{y(1)}\right)^2\right)\frac{1}{t}+\left(\frac{y'(1)}{y(1)}\right)^2-\frac{y'(1)}{y(1)}.
$$
Remark that we choose $c = c_\nu(m)$ as the a number satisfying the conditions
$$
\frac{y'(1)}{y(1)} = \frac{1}{2}+c\frac{J'_\nu\left(\frac{2}{m}c\right)}{J_\nu\left(\frac{2}{m}c\right)} = \lambda.
$$
Therefore,  $\frac{1+\nu m}{2}-\lambda^2\geq 0$, $\lambda^2-\lambda \leq 0$ and
$$
\int\limits_0^1 \frac{|f(x)|}{x^2} \left( \frac{1-\nu^2 m^2}{4}+c_\nu^2(m)x^{m} \right) dx \leq
$$
$$
\leq\left(\frac{1+\nu m}{2}-\lambda^2\right)\int\limits_0^1 \frac{|f'(t)|}{t} dt+\left(\lambda^2-\lambda\right)\int\limits_0^1|f'(t)| dt.
$$
\textbf{Case 2:} Exist a number $x_0$ such that $y'(x_0)=0$.
In this case, since $xy'(x)/y(x)$ is decreasing, we have that
$$
\min_{x\in[0,1]}\left(\frac{xy'(x)}{y(x)}\right)^2 =0
$$
and
$$
T(t)  = -\int\limits_t^1 \left(\frac{y'(x)}{y(x)}\right)'+ \left(\frac{y'(x)}{y(x)}\right)^2 dx \leq \frac{y'(t)}{y(t)}-\frac{y'(1)}{y(1)}.
$$
As before,
$$
\int\limits_0^1 \frac{|f(x)|}{x^2} \left( \frac{1-\nu^2 m^2}{4}+c_\nu^2(m)x^{m} \right) dx\leq \int\limits_0^1 \frac{|f'(t)|}{t}\left(\frac{1+\nu m}{2}-\lambda t\right) dt.
$$
This completes the proof of Lemma \ref{l3}.
\end{proof}
Note that in  \cite{AW_ZAMM} the authors considered the condition $\lambda = 0$. In this case the application of Lemma \ref{l3} yields the following corollary.
\begin{cor}
Suppose that $m>0$, and let $f$ be an absolutely continuous function in $[0, 1]$ such that $f(0) = 0$ and $f'/x\in L^1[0, 1]$. If  $\nu\in (0,1/m]$, then
$$
\frac{1-\nu^2 m^2}{4}\int\limits_0^1 \frac{|f(x)|}{x^2} dx +c^2_\nu(m)\int\limits_0^1 \frac{|f(x)|}{x^{2-m}} dx \leq \frac{1+\nu m}{2} \int\limits_0^1 \frac{|f'(x)|}{x} dx,
$$
where   $c= c(\nu,m)$  is a number satisfying the conditions
$$
J_\nu\left(\frac{2}{m}c\right)+2c J'_\nu\left(\frac{2}{m}c\right) = 0 \quad\text{and}\quad \frac{2}{m}c\in(0,j_{\nu}).
$$
\end{cor}
\begin{cor}\label{cor2}
 Suppose that $m>1$, and let $f$ be an absolutely continuous function in $[0, 1]$ such that $f(0) = 0$ and $f'/x\in L^1[0, 1]$. If  $\nu\in (0,1/m]$ and $\lambda\in\left[0,\frac{1+\nu m}{2}\right)$, then
\begin{multline}\nonumber
\frac{1-\nu^2 m^2}{4}\int\limits_0^1 \frac{|f(x)|}{x^2} dx +\left(c^2_\nu(m)+(m-1)(\lambda-\lambda^2)\right)\int\limits_0^1 \frac{|f(x)|}{x^{2-m}} dx \leq\\
 \left(\frac{1+\nu m}{2}-\lambda^2\right) \int\limits_0^1 \frac{|f'(x)|}{x} dx,
\end{multline}
where   $c= c_\nu(m)$  is a number satisfying the conditions
$$
(1-2\lambda)J_\nu\left(\frac{2}{m}c\right)+2cJ'_\nu\left(\frac{2}{m}c\right) = 0 \quad\text{and}\quad\frac{2}{m}c\in(0,j_{\nu}).
$$
\end{cor}
\begin{proof} It is known  \cite{Avkh_Nas_SibMJ_2014} that if $m>1$ and $f$ is an absolutely continuous function in $[0, 1]$ such that $f(0) = 0$, then the following sharp inequality
 $$
(m-1)\int\limits_{0}^1 \frac{|f(x)|}{x^{2-m}} dx < \int\limits_{0}^1|f'(x)|dx
$$
is valid. Using this inequality and Lemma \ref{l3}, we get Corollary \ref{cor2}.
\end{proof}
The following theorem is extension of Lemma \ref{l3} for arbitrary segment cases.
\begin{thm}\label{th1}
 Suppose that $m>0$, $\nu\in (0,1/m]$, and let $f$ be an absolutely continuous function in $[a, b]$ such that $f(a)=f(b) = 0$ and $f'/x\in L^1[a, b]$. If  $\lambda\in\left[0,\frac{1+\nu m}{2}\right)$, then
\begin{multline}\nonumber
\frac{1-\nu^2 m^2}{4}\int\limits_a^b \frac{|f(x)|}{\delta(x)^2} dx +\frac{c_\nu^2(m)}{\delta_0^m}\int\limits_a^b \frac{|f(x)|}{\delta(x)^{2-m}} dx \leq\\
\left(\frac{1+\nu m}{2}-\lambda^2\right) \int\limits_a^b \frac{|f'(x)|}{\delta(x)} dx+\frac{\lambda^2-\lambda}{\delta_0}\int\limits_a^b|f'(x)| dx,
\end{multline}
and if $ \lambda\leq 0$, then
\begin{multline}\nonumber
\frac{1-\nu^2 m^2}{4}\int\limits_a^b \frac{|f(x)|}{\delta(x)^2} dx +\frac{c_\nu^2(m)}{\delta_0^m}\int\limits_a^b \frac{|f(x)|}{\delta(x)^{2-m}} dx \leq\\
 \frac{1+\nu m}{2}\int\limits_a^b \frac{|f'(x)|}{\delta(x)}dx-\frac{\lambda}{\delta_0}\int\limits_a^b |f'(x)|dx,
\end{multline}
where $\delta(x) = \min\{b-x,x-a\}$, $\delta_0= \frac{b-a}{2}$ and the constant  $c_\nu(m)$  is a number satisfying the conditions
$$
(1-2\lambda)J_\nu\left(\frac{2}{m}c_\nu(m)\right)+2c_\nu(m)J'_\nu\left(\frac{2}{m}c_\nu(m)\right)=0\quad\text{and}\quad\frac{2}{m}c_\nu(m)\in(0,j_{\nu}).
$$
\end{thm}
\begin{proof}For an arbitrary $\rho> 0$, the change of variable $x = \rho \tau$ in the first inequality in Lemma \ref{l3} leads us to
\begin{multline}\nonumber
\frac{1-\nu^2 m^2}{4}\int\limits_0^\rho \frac{|f(x)|}{x^2} dx +\frac{c_\nu^2(m)}{\rho^m}\int\limits_0^\rho \frac{|f(x)|}{x^{2-m}} dx \leq\\
 \left(\frac{1+\nu m}{2}-\lambda^2\right) \int\limits_0^\rho \frac{|f'(x)|}{x} dx+\frac{\lambda^2-\lambda}{\rho}\int\limits_0^\rho|f'(x)| dx.
\end{multline}
Applying the latter inequality to the functions $f(\tau) = g(\tau + a)$ and  $u(\tau) = f(b- \tau)$ with
$ \rho= \delta_0 = (b-a)/2$, we get
\begin{multline}\nonumber
\frac{1-\nu^2 m^2}{4}\int\limits_a^{(a+b)/2} \frac{|f(x)|}{(x-a)^2} dx +\frac{c_\nu^2(m)}{\delta_0^m}\int\limits_a^{(a+b)/2} \frac{|f(x)|}{(x-a)^{2-m}} dx \leq\\
 \left(\frac{1+\nu m}{2}-\lambda^2\right) \int\limits_a^{(a+b)/2} \frac{|f'(x)|}{x-a} dx+\frac{\lambda^2-\lambda}{\delta_0}\int\limits_a^{(a+b)/2}|f'(x)| dx,
\end{multline}
and
\begin{multline}\nonumber
\frac{1-\nu^2 m^2}{4}\int\limits^b_{(a+b)/2} \frac{|f(x)|}{(b-x)^2} dx +\frac{c_\nu^2(m)}{\delta_0^m}\int\limits^b_{(a+b)/2} \frac{|f(x)|}{(b-x)^{2-m}} dx \leq\\
 \left(\frac{1+\nu m}{2}-\lambda^2\right) \int\limits^b_{(a+b)/2} \frac{|f'(x)|}{b-x} dx+\frac{\lambda^2-\lambda}{\delta_0}\int\limits^b_{(a+b)/2}|f'(x)| dx.
\end{multline}
The sum of these two inequalities gives the required statement. The second inequality of Theorem \ref{th1} is proved in the same way.

 This completes the proof of Theorem \ref{th1}.
\end{proof}
We will give some examples.

\textbf{Example 1.} If $a=-1, b=1, \nu = 1, m=1$ and $\lambda =\frac{1}{2}$, we get
$$
\frac{j_1'^2}{3}\int\limits_{-1}^1 \frac{|f(x)|}{1-|x|} dx \leq \int\limits_{-1}^1 \frac{|f'(x)|}{1-|x|} dx-\frac{1}{3}\int\limits_{-1}^1|f'(x)| dx,
$$
where ${j'}_1$ is the first positive zero of the derivative $J'_1$ of Bessel's function $J_1$ and it is known that $j'_1 \approx 1,8412$ (see \cite[p. 411, t. 95]{Table}).

For comparison, we give the following sharp inequality from \cite{Avkh_Nas_SibMJ_2014}.
$$
e\int\limits_{-1}^1 \frac{|f(x)|}{1-|x|} dx \leq \int\limits_{-1}^1 \frac{|f'(x)|}{1-|x|} dx.
$$

\textbf{Example 2.} If $a=-1, b=1, \nu = 1$ and $m=1$, then
$$
\frac{c^2}{1-\lambda^2}\int\limits_{-1}^{1} \frac{|f(x)|}{1-|x|} dx \leq\int\limits_{-1}^1 \frac{|f'(x)|}{1-|x|} dx+\frac{\lambda^2-\lambda}{1-\lambda^2}\int\limits_{-1}^1|f'(x)| dx,
$$
where $c$  is a number satisfying the conditions
$$
\frac{cJ_0(2c)}{J_1(2c)}= \lambda, \quad 2c\in(0,j_{\nu}) \quad\text{and}\quad  c\to 0 \text{  as  } \lambda\to 1.
$$
In the limit when $\lambda \to 1$, since
$$
\lim_{c_\nu\to 0}\frac{c^2}{1-c^2\frac{J_{0}^2(2c)}{J_1^2(2c)}}=1 \quad\text{and}\quad \lim_{\lambda\to 1}\frac{\lambda^2-\lambda}{1-\lambda^2}=-\frac{1}{2},
$$
the last inequality implies
$$
\int\limits_{-1}^1 \frac{|f(x)|}{1-|x|} dx \leq \int\limits_{-1}^1 \frac{|f'(x)|}{1-|x|} dx-\frac{1}{2}\int\limits_{-1}^1|f'(x)| dx.
$$

\textbf{Example 3.} If $a=-1, b=1, \nu = 0, m=1$ and $\lambda \to 0.5$, we get
$$
\int\limits_{-1}^1 \frac{|f(x)|}{(1-|x|)^2} dx \leq \int\limits_{-1}^1 \frac{|f'(x)|}{1-|x|} dx-\int\limits_{-1}^1|f'(x)| dx.
$$
Above we use that $\lim_{\lambda\to 1/2}c_0(1)=0$.

For comparison, we give the following sharp inequality from \cite{Av4} (see also \cite{Nas_LJM_2019}, Corollary 1.).
$$
\int\limits_{-1}^1 \frac{|f(x)|}{(1-|x|)^2} dx \leq \int\limits_{-1}^1 \frac{|f'(x)|}{1-|x|} dx-\int\limits_{-1}^1|f'(x)|(1-|x|) dx.
$$

\subsection{$L_p$ inequalities.} In this section we establish $L_p$ inequalities. We consider inequalities for $p\geq 1$ and $p=2$. To prove $L_p$ inequalities for $p>1$ we will use H\"{o}lder's inequality and the following inequalities of Opial type. For more information about Opial's inequalities we refer to \cite{Shum}. If $f$ is absolutely continuous on $(0,\rho)$ with $f(0) = 0$, then the following Opial inequalities hold:
\begin{equation}\label{opil_in_1}
\int\limits_0^\rho \frac{|f(x)|f'(x)|}{x}dx\leq 2\int\limits_0^\rho |f'(x)|^{2}dx
\end{equation}
and
\begin{equation}\label{opil_in_2}
\int\limits_0^\rho |f(x)||f'(x)|dx\leq \frac{\rho}{2}\int\limits_0^\rho |f'(x)|^{2}dx.
\end{equation}
The following theorem holds.
\begin{thm}  \label{th2}Suppose that $p\geq 1, r\in [1,p]$ and $f$ be an absolutely continuous function in $[a, b]$ such that $f(a)=f(b) = 0$. If $m>1, \nu\in [0,1/m]$ and $\lambda\in\left[0,\frac{1+\nu m}{2}\right)$, then
\begin{multline}\nonumber
(1-r\nu^2 m^2)\int\limits_a^b \frac{|g(x)|^p}{\delta(x)^2} dx +\frac{4r\left(c_\nu^2(m)+(m-1)(\lambda -\lambda ^2)\right)}{\delta_0^m}\int\limits_a^b \frac{|g(x)|^p}{\delta(x)^{2-m}} dx \leq \\
p^r\left(2(1+\nu m)-4\lambda^2\right)^r \int\limits_a^b \frac{|g(x)|^{p-r}|g'(x)|^{r}}{\delta(x)^{2-r}} dx
\end{multline}
and if $m>1, \nu\in [0,1/m)$ and $\lambda\in\left[0,\frac{1+\nu m}{2}\right)$, then
\begin{multline}\nonumber
\int\limits_a^b \frac{|g(x)|^p}{\delta(x)^2} dx +\frac{4r\left(c_\nu^2(m)+(m-1)(\lambda -\lambda^2)\right)}{(1-\nu^2 m^2)\delta_0^m}\int\limits_a^b \frac{|g(x)|^p}{\delta(x)^{2-m}} dx \leq \\
 p^r \left(\frac{2}{1-\nu m}-\frac{4\lambda^2}{1-\nu^2m^2}\right)^r\int\limits_a^b \frac{|g(x)|^{p-r}|g'(x)|^{r}}{\delta(x)^{2-r}} dx,
\end{multline}
where $\delta(x) = \min\{b-x,x-a\}$, $\delta_0= \frac{b-a}{2}$ and the constant  $c_\nu(m)$  is a number satisfying the conditions
$$
(1-2\lambda)J_\nu\left(\frac{2}{m}c_\nu(m)\right)+c_\nu(m)J'_\nu\left(\frac{2}{m}c_\nu(m)\right) = 0\quad\text{and}\quad\frac{2}{m}c_\nu(m)\in(0,j_{\nu}).
$$
\end{thm}
\begin{proof} Let $\rho>0$, and let a function $g\in C^1[0,\rho]$. Then $f(x)=|g(x)|^p$ belongs to $C^1[0,\rho]$, because $$
\frac{d}{dx}|g(x)|^p = p|g(x)|^{p-1}g'(x)\textrm{sign} g(x)
$$
and the function $|g(x)|^{p-1}\textrm{sign} g(x)$ is continuous for $p>1$.

The application of Corollary \ref{cor2} to  the function $f(x)=|g(x)|^p\in C^1[0,\rho]$ yields
\begin{multline}\nonumber
(1-\nu^2 m^2)\int\limits_0^1 \frac{|g(x)|^p}{x^2} dx +4\left(c_\nu^2(m)+(m-1)(\lambda -\lambda ^2)\right)\int\limits_0^1 \frac{|g(x)|^p}{x^{2-m}} dx \leq\\
4p\left(\frac{1+\nu m}{2}-\lambda^2\right) \int\limits_0^1 \frac{|g'(x)||g(x)|^{p-1}}{x} dx.
\end{multline}
Using the convex inequality(see \cite[p. 37]{HLP})
\begin{equation}\label{el_ineq}
a^{p_1}b^{p_2}\leq \left(\frac{p_1a+p_2b}{p_1+p_2}\right)^{p_1+p_2}
\end{equation}
to the quantities
$$
a = \frac{|g(x)|^p}{x^2}, b = p^r\left(2(1+\nu m)-4\lambda^2\right)^r \frac{|g(x)|^{p-r}|g'(x)|^{r}}{x^{2-r}}, p_1 = 1-\frac{1}{r}, p_2  = \frac{1}{r},
$$
we get
\begin{multline}\nonumber
(1-r\nu^2 m^2)\int\limits_0^1 \frac{|g(x)|^p}{x^2} dx +4r\left(c_\nu^2(m)+(m-1)(\lambda -\lambda ^2)\right)\int\limits_0^1 \frac{|g(x)|^p}{x^{2-m}} dx \leq\\
 p^r\left(2(1+\nu m)-4\lambda^2\right)^r \int\limits_0^1 \frac{|g(x)|^{p-r}|g'(x)|^{r}}{x^{2-r}} dx.
\end{multline}
Taking into account  (\ref{el_ineq}) to the quantities
$$
a = \frac{|g(x)|^p}{x^2}, b = p^r\left(\frac{2}{1-\nu m}-\frac{4\lambda^2}{1-\nu^2m^2}\right)^r\frac{|g(x)|^{p-r}|g'(x)|^{r}}{x^{2-r}}, p_1 = 1-\frac{1}{r}, p_2  = \frac{1}{r},
$$
we get
\begin{multline}\nonumber
\int\limits_0^1 \frac{|g(x)|^p}{x^2} dx +\frac{4r\left(c_\nu^2(m)+(m-1)(\lambda -\lambda ^2)\right)}{1-\nu^2 m^2}\int\limits_0^1 \frac{|g(x)|^p}{x^{2-m}} dx \leq\\
p^r \left(\frac{2}{1-\nu m}-\frac{4\lambda^2}{1-\nu^2m^2}\right)^r\int\limits_0^1 \frac{|g(x)|^{p-r}|g'(x)|^{r}}{x^{2-r}} dx.
\end{multline}
Hence, we have inequalities on $[0,1]$. As before, to get inequalities in arbitrary segment $[a,b]$ we use change of variables.

This concludes proof of Theorem \ref{th2}.
\end{proof}
If we put that in Theorem \ref{th2} $r=p$ and
$$
2(1+\nu m)-4\lambda^2 \leq 1, \frac{2}{1-\nu m}-\frac{4\lambda^2}{1-\nu^2m^2}=1
$$
 we obtain the following assertion.

\begin{cor} Suppose that $ p\geq 1$, $m>1$, $\nu\in [0,1/m]$, and $f$ be an absolutely continuous function in $[a, b]$ such that $f(a)=f(b) = 0$. If $\lambda\in\left[\frac{\sqrt{1+2\nu m}}{4},\frac{1+\nu m}{2}\right)$, then
\begin{multline}\nonumber
 \left(c_\nu^2(m)+(m-1)(\lambda -\lambda ^2)\right)\frac{4p}{\delta_0^m}\int\limits_a^b \frac{|g(x)|^p}{\delta(x)^{2-m}} dx\\
 \leq p^p \int\limits_a^b \frac{|g'(x)|^{p}}{\delta(x)^{2-p}} dx-(1-p\nu^2 m^2)\int\limits_a^b \frac{|g(x)|^p}{\delta(x)^2} dx
\end{multline}
and if  $\lambda=\frac{1+\nu m}{2}$, then
$$
\int\limits_a^b \frac{|g(x)|^p}{\delta(x)^2} dx +\frac{p\left(m-1\right)}{\delta_0^m}\int\limits_a^b \frac{|g(x)|^p}{\delta(x)^{2-m}} dx
\leq p^p \int\limits_a^b \frac{|g'(x)|^{p}}{\delta(x)^{2-p}} dx,
$$
where $\delta(x) = \min\{b-x,x-a\}$, $\delta_0= \frac{b-a}{2}$ and the constant  $c_\nu(m)$  is a number satisfying the conditions
$$
(1-2\lambda)J_\nu\left(\frac{2}{m}c_\nu(m)\right)2+c_\nu(m)J'_\nu\left(\frac{2}{m}c_\nu(m)\right) = 0\quad\text{and}\quad\frac{2}{m}c_\nu(m)\in(0,j_{\nu}).
$$
\end{cor}
Now, we establish other $L_2$ inequalities. Namely, the following theorem is valid.

\begin{thm} \label{th3}Suppose that $p\geq 1, r\in [1,p]$, and let $f$ be an absolutely continuous function in $[a, b]$ such that $f(a)=f(b) = 0$. If $m>1, \nu\in [0,1/m]$ and $\lambda\in\left[0,\frac{1+\nu m}{2}\right)$, then
\begin{multline}\nonumber
\frac{1-\nu^2 m^2}{4}\int\limits_a^b \frac{|g(x)|^2}{\delta(x)^2} dx +\frac{c_\nu^2(m)+(m-1)(\lambda-\lambda^2)}{\delta_0^m}\int\limits_a^b \frac{|g(x)|^2}{\delta(x)^{2-m}} dx\leq\\
 4\left(\frac{1+\nu m}{2}-\lambda^2\right) \int\limits_a^b|g'(x)|^2 dx,
\end{multline}
and if  $m>0, \nu\in [0,1/m)$ and $\lambda\leq 0$, then
$$
\frac{1-\nu^2 m^2}{4}\int\limits_a^b \frac{|g(x)|^2}{\delta(x)^2} dx +\frac{c_{\nu}^2(m)}{\delta_0^m}\int\limits_a^b \frac{|g(x)|^2}{\delta(x)^{2-m}} dx \leq \left(2(1+\nu m)+|\lambda|\right) \int\limits_a^b|g'(x)|^2 dx,
$$
where $\delta(x) = \min\{b-x,x-a\}$, $\delta_0= \frac{b-a}{2}$ and the constant  $c_\nu(m)$  is a number satisfying the conditions
$$
(1-2\lambda)J_\nu\left(\frac{2}{m}c_\nu(m)\right)+2c_\nu(m)J'_\nu\left(\frac{2}{m}c_\nu(m)\right) = 0 \quad\text{and}\quad\frac{2}{m}c_\nu(m)\in(0,j_{\nu}).
$$
\end{thm}
\begin{proof} The application of Corollary \ref{cor2} to  the function $f(x)=|g(x)|^2\in C^1[0,\rho]$ yields
\begin{multline}\nonumber
\frac{1-\nu^2 m^2}{4}\int\limits_0^1 \frac{|g(x)|^2}{x^2} dx +(c_\nu^2(m)+(m-1)(\lambda-\lambda^2))\int\limits_0^1 \frac{|g(x)|^2}{x^{2-m}} dx \leq\\
2\left(\frac{1+\nu m}{2}-\lambda^2\right) \int\limits_0^1\frac{|g(x)||g'(x)|}{x} dx.
\end{multline}
Combining the last inequality and (\ref{opil_in_1}), we get
\begin{multline}\nonumber
\frac{1-\nu^2 m^2}{4}\int\limits_0^1 \frac{|g(x)|^2}{x^2} dx +(c_\nu^2(m)+(m-1)(\lambda-\lambda^2))\int\limits_0^1 \frac{|g(x)|^2}{x^{2-m}} dx \leq\\
 4\left(\frac{1+\nu m}{2}-\lambda^2\right) \int\limits_0^1|g'(x)|^2 dx.
\end{multline}
The application of lemma \ref{l3} to  the function $f(x)=|g(x)|^2\in C^1[0,\rho]$ yields
\begin{multline}\nonumber
\frac{1-\nu^2 m^2}{4}\int\limits_0^1 \frac{|g(x)|^2}{x^2} dx +c_{\nu}^2(m)\int\limits_0^1 \frac{|g(x)|^2}{x^{2-m}} dx\\
2\left(\frac{1+\nu m}{2}\right) \int\limits_0^1\frac{|g(x)||g'(x)|}{x} dx+2|\lambda| \int\limits_0^1|g(x)||g'(x)| dx,
\end{multline}
where  $\lambda\leq 0$. Hence, using a(\ref{opil_in_1}) and (\ref{opil_in_2}), we obtain
$$
\frac{1-\nu^2 m^2}{4}\int\limits_0^1 \frac{|g(x)|^2}{x^2} dx +c_{\nu}^2(m)\int\limits_0^1 \frac{|g(x)|^2}{x^{2-m}} dx \leq \left(2(1+\nu m)+|\lambda|\right) \int\limits_0^1|g'(x)|^2 dx.
$$
\end{proof}
\section{Inequalities in convex domains of the Euclidean space $\mathbb{R}^n$. }\label{sec2}
This section is devoted to multidimensional analogues. We consider inequalities in a convex domain  $\Omega$ with finite inner radius.
To prove inequalities in domains we use Avkhadiev F.G. method from \cite{Av3} (see also \cite{Avkh_Nas_SibMJ_2014}, \cite{Av4}). This method based on special approximations and partitions of the integration domain $\Omega$ by cubes. In the case of convex domains the proof multidimensional Hardy inequalities is reduced to the application of one-dimensional inequalities (see \cite{AW_ZAMM} and \cite{BEL}).

Let $\Omega$ is  an open and convex set in $\mathbb{R}^n$ with finite inradius
$$
\delta_0= \delta_0(\Omega) = \sup_{z\in \Omega}\delta(x),
$$
where $\delta(x) = dist (x,\partial\Omega)$. By $C_0^1(\Omega)$ denote the family of continuously differentiable functions $f :\Omega \to \mathbb{R}$ with compact supports
lying in $\Omega$.

As we mentioned above, for a convex domain the situation is simple and one-dimensional inequalities are extended straightforwardly to the spatial case. Namely,  Avkhadiev's method is reduced to the following statement.

\textbf{Theorem A.}\emph{ Let $\Omega$ be an open, convex set in $\mathbb{R}^n$ with finite inradius $\delta_0 = \delta_0(\Omega)$. If, for $\alpha\in(0,\delta_0]$ and nonnegative constancs $b,c$
$$
\int\limits_0^\alpha \frac{|f'(t)|^p}{t^{s-p}}dt\geq b^2\int\limits_0^\alpha \frac{|f(t)|^p}{t^s}dt+\frac{c^2}{\delta_0^m}\int\limits_0^\alpha \frac{|f(t)|^p}{t^{s-m}}dt, \ \  \ f\in C_0^1(0,2\alpha),
$$
then
$$
\int\limits_\Omega\frac{ |\nabla f(x)|^p}{\delta(x)^{s-p}}dx\geq b^2\int\limits_\Omega \frac{|f(x)|^p}{\delta(x)^s}dx+\frac{c^2}{\delta_0^m}\int\limits_\Omega \frac{|f(x)|^p}{\delta(x)^{s-m}}dx,  \ \  \ f\in C_0^1(\Omega).
$$
}
\subsection{$L_1$ inequalities.}
Combining Theorem A and Theorem \ref{th1}, we obtain the following statement.
\begin{thm}
 Let $\Omega$ be an open, convex set in $\mathbb{R}^n$ with finite inradius $\delta_0 = \delta_0(\Omega)$. Suppose that $m>0$, $\nu\in (0,1/m]$, $f\in C_0^1(\Omega)$ and $|\nabla f|/\delta(\cdot)\in L^1(\Omega)$. If  $\nu\in (0,1/m]$ and $\lambda\in\left[0,\frac{1+\nu m}{2}\right)$, then
 \begin{multline}\nonumber
\frac{1-\nu^2 m^2}{4}\int\limits_\Omega \frac{|f(x)|}{\delta(x)^2} dx +\frac{c_\nu^2(m)}{\delta_0^m}\int\limits_\Omega \frac{|f(x)|}{\delta(x)^{2-m}} dx \leq\\
\left(\frac{1+\nu m}{2}-\lambda^2\right) \int\limits_\Omega\frac{|\nabla f(x)|}{\delta(x)} dx+\frac{\lambda^2-\lambda}{\delta_0}\int\limits_\Omega|\nabla f (x)| dx,
\end{multline}
and if $ \lambda\leq 0$, then
\begin{multline}\nonumber
\frac{1-\nu^2 m^2}{4}\int\limits_\Omega \frac{|f(x)|}{\delta(x)^2} dx +\frac{c_\nu^2(m)}{\delta^m_0}\int\limits_\Omega \frac{|f(x)|}{\delta(x)^{2-m}} dx\leq\\
  \frac{1+\nu m}{2}\int\limits_\Omega \frac{|\nabla f(x)|}{\delta(x)}dx-\frac{\lambda}{\delta_0}\int\limits_\Omega |\nabla f(x)|dx,
\end{multline}
where the constant  $c_\nu(m)$  is a number satisfying the conditions
$$
(1-2\lambda)J_\nu\left(\frac{2}{m}c_\nu(m)\right)+2c_\nu(m)J'_\nu\left(\frac{2}{m}c_\nu(m)\right) = 0\quad\text{and}\quad\frac{2}{m}c_\nu(m)\in(0,j_{\nu}).
$$
\end{thm}

Combining Theorem A and Corollary \ref{cor2}, we obtain the following theorem.

\begin{thm}
  Let $\Omega$ be an open, convex set in $\mathbb{R}^n$ with finite inradius $\delta_0 = \delta_0(\Omega)$. Suppose that $m>1$, $f\in C_0^1(\Omega)$ and $|\nabla f|/\delta(x)\in L^1(\Omega)$. If  $\nu\in (0,1/m]$ and $\lambda\in\left[0,\frac{1+\nu m}{2}\right)$, then
\begin{multline}\nonumber
\frac{1-\nu^2 m^2}{4}\int\limits_\Omega \frac{|f(x)|}{\delta(x)^2} dx +\frac{c^2_\nu(m)+(m-1)(\lambda-\lambda^2)}{\delta_0^m}\int\limits_\Omega\frac{|f(x)|}{\delta(x)^{2-m}} dx \leq\\
 \left(\frac{1+\nu m}{2}-\lambda^2\right) \int\limits_\Omega \frac{|\nabla f(x)|}{\delta(x)} dx,
\end{multline}
where   $c= c_\nu(m)$  is a number satisfying the conditions
$$
(1-2\lambda)J_\nu\left(\frac{2}{m}c_\nu(m)\right)+2c_\nu(m)J'_\nu\left(\frac{2}{m}c_\nu(m)\right) = 0\quad\text{and}\quad\frac{2}{m}c_\nu(m)\in(0,j_{\nu}).
$$
\end{thm}
\subsection{$L_p$ inequalities.}
Combining Theorem \ref{th2}, Theorem \ref{th3}  and Theorem A., we obtain spatial analogues of the one-dimensional inequalities from section 2.2. Namely, we get

\begin{thm}  Let $\Omega$ be an open, convex set in $\mathbb{R}^n$ with finite inradius $\delta_0 = \delta_0(\Omega)$. Suppose that $p\geq 1$ and $f\in C_0^1(\Omega)$. If $m>1, \nu\in [0,1/m]$ and $\lambda\in\left[0,\frac{1+\nu m}{2}\right)$, then
\begin{multline}\nonumber
(1-p\nu^2 m^2)\int\limits_\Omega \frac{|g(x)|^p}{\delta(x)^2} dx +\frac{4p\left(c_\nu^2(m)+(m-1)(\lambda -\lambda^2)\right)}{\delta_0^m}\int\limits_\Omega \frac{|g(x)|^p}{\delta(x)^{2-m}} dx \leq\\
 p^p\left(2(1+\nu m)-4\lambda^2\right)^p \int\limits_\Omega \frac{|\nabla g(x)|^{p}}{\delta(x)^{2-p}} dx,
\end{multline}
and if $m>1, \nu\in [0,1/m)$ and $\lambda\in\left[0,\frac{1+\nu m}{2}\right)$, then
\begin{multline}\nonumber
\int\limits_\Omega \frac{|g(x)|^p}{\delta(x)^2} dx +\frac{4p\left(c_\nu^2(m)+(m-1)(\lambda -\lambda ^2)\right)}{(1-\nu^2 m^2)\delta_0^m}\int\limits_\Omega \frac{|g(x)|^p}{\delta(x)^{2-m}} dx \leq\\
p^p \left(\frac{2}{1-\nu m}-\frac{4\lambda^2}{1-\nu^2m^2}\right)^p\int\limits_\Omega \frac{|g'(x)|^{p}}{\delta(x)^{2-p}} dx,
\end{multline}
where  the constant  $c_\nu(m)$  is a number satisfying the conditions
$$
(1-2\lambda)J_\nu\left(\frac{2}{m}c_\nu(m)\right)+2c_\nu(m)J'_\nu\left(\frac{2}{m}c_\nu(m)\right) = 0\quad\text{and}\quad\frac{2}{m}c_\nu(m)\in(0,j_{\nu}).
$$
\end{thm}
\begin{thm}\label{th7} Let $\Omega$ be an open, convex set in $\mathbb{R}^n$ with finite inradius $\delta_0 = \delta_0(\Omega)$. Suppose that $f\in C_0^1(\Omega)$ and $|\nabla f|\in L^2(\Omega)$. If $m>1, \nu\in [0,1/m]$ and $\lambda\in\left[0,\frac{1+\nu m}{2}\right)$, then
\begin{multline}\nonumber
\frac{1-\nu^2 m^2}{4}\int\limits_\Omega \frac{|g(x)|^2}{\delta(x)^2} dx +\frac{c_\nu^2(m)+(m-1)(\lambda-\lambda^2)}{\delta_0^m}\int\limits_\Omega \frac{|g(x)|^2}{\delta(x)^{2-m}} dx \leq\\
 4\left(\frac{1+\nu m}{2}-\lambda^2\right) \int\limits_\Omega|\nabla g(x)|^2 dx,
\end{multline}
and if  $m>0, \nu\in [0,1/m)$ and $\lambda\leq 0$, then
$$
\frac{1-\nu^2 m^2}{4}\int\limits_\Omega \frac{|g(x)|^2}{\delta(x)^2} dx +\frac{c_{\nu}^2(m)}{\delta_0^m}\int\limits_\Omega \frac{|g(x)|^2}{\delta(x)^{2-m}} dx \leq \left(2(1+\nu m)+|\lambda|\right) \int\limits_\Omega|\nabla g(x)|^2 dx,
$$
where  the constant  $c_\nu(m)$  is a number satisfying the conditions
$$
(1-2\lambda)J_\nu\left(\frac{2}{m}c_\nu(m)\right)+2c_\nu(m)J'_\nu\left(\frac{2}{m}c_\nu(m)\right) = 0\quad\text{and}\quad\frac{2}{m}c_\nu(m)\in(0,j_{\nu}).
$$
\end{thm}
\begin{cor}\label{cor5}
If  $m>1, \nu\in [0,1/m]$, $\lambda\in\left[0,\frac{1+\nu m}{2}\right)$ and $ 4\left(\frac{1+\nu m}{2}-\lambda^2\right) \leq 1$, then the first constant $(1-\nu^2 m^2)/4$ is sharp and
$$
c_\nu^2(m)+(m-1)(\lambda-\lambda^2)\leq  C_\nu^2(m).
$$
\end{cor}
\begin{proof} This corollary immediately follows from Theorem \ref{th7} and sharp inequality~(\ref{intr_f1}).
\end{proof}
\bigskip

This work was supported be the grant of the President of the Russian Federation  MK-709.2019.1.


\begin{thebibliography}{99}
%
%

\bibitem{Table} \textsc{M. Abramowitz, I.A. Stegun}, Handbook of mathematical functions with formulas, graphs, and mathematical tables. USA: National Bureau of Standards Applied Mathematics Series (1964).
\bibitem{APR} \textsc{V.~Alvarez, D.~Pestana, J.~M.~Rodr\'{i}guez}, Isoperimetric inequalities in Riemann surfaces of infinite type, Revista Matem\`{A}tica Iberoamericana, \textbf{15}:2 (1999), 353--425.
\bibitem{AvLap} \textsc{F.~Avkhadiev, A.~Laptev}, Hardy inequalities for nonconvex domains, Around the research of Vladimir Maz'ya, v. I, Int. Math. Ser. (N. Y.), 11, Springer, New York, 2010, 1-12.
\bibitem{Avkh_IM} \textsc{F.~G.~Avkhadiev}, A geometric description of domains whose Hardy constant is equal to 1/4, ~Izvestiya: Mathematics \textbf{78}:5 (2014), 855--876.
\bibitem{AW_ZAMM} \textsc{F.~G.~Avkhadiev, K.-J.~Wirths},  Unified  Poincar\'{e} and Hardy inequalities with sharp constants for convex domains, Z. Angew. Math. Mech. \textbf{87} (2007), 632--642.
\bibitem{AW_Lamb} \textsc{F.~G.~Avkhadiev, K.-J.~Wirths}, Sharp Hardy-type inequalities with Lamb's constants,  Bull. Belg. Math. Soc. Simon Stevin. \textbf{18} (4) (2011), 723--736.
\bibitem{Avkh_Nas_SibMJ_2014} \textsc{F.G. Avkhadiev, R.G. Nasibullin}, Hardy-type inequalities in arbitrary domains with finite inner radius. Siberian Mathematical Journal. \textbf{55}:2 (2014), 191-200.
\bibitem{Avkh_AIA} \textsc{F.~G.~Avkhadiev}, Integral inequalities of Hardy and Rellich in domains satisfying an exterior sphere condition, Algebra i Analiz, \textbf{30}:2 (2018), 18--44.	 	


\bibitem{Avkh_JMA} \textsc{F.~G. Avkhadiev},  Hardy-Rellich inequalities in domains of the Euclidean space, J. Math Anal. Appl. \textbf{442} (2016), 469--484.

\bibitem{Av3} \textsc{F.~G.~Avkhadiev}, Hardy type inequalities in higher dimensions with explicit estimate of  constants, Lobachevskii J. Math., \textbf{21} (2006), 3--31

\bibitem{Av4} \textsc{F.~G.~Avkhadiev}, Hardy-type inequalities on planar and spatial open sets. Proceedings of the Steklov Institute of Mathematics, \textbf{255}:1 (2006), 2--12


\bibitem{AvkhMS_15} \textsc{F.~G.~Avkhadiev}, Integral inequalities in domains of hyperbolic type and their applications, Sb. Math. \textbf{206}:12 (2015), 1657--1681

\bibitem{ANSh_2018} \textsc{F.~G.~Avkhadiev, R.~G.~Nasibullin, I.~K.~Shafigullin}, Lp-Versions of One Conformally Invariant Inequality, Russian Mathematics volume, \textbf{62} (2018), 76--79.

\bibitem{ANSh_2019} \textsc{F.~G.~Avkhadiev, R.~G.~Nasibullin, I.~K.~Shafigullin}, Conformal invariants of hyperbolic planar domains, Ufa Math. J. \textbf{11}:2 (2019), 3--18.

\bibitem{BEL}  \textsc{A.~A.~Balinsky, W.~D.~Evans, R.~T.~Lewis} \emph{The Analysis and Geometry of Hardy's Inequality}, Heidelberg - New York - Dordrecht - London: Universitext, Springer, 2015.

\bibitem{BM} \textsc{H.~Brezis, M.~Marcus}, Hardy's inequalities revisited, Dedicated to E. De Giorgi, Ann. Scuola Norm. Sup. Pisa Cl. Sci. (4), \textbf{25}:1-2 (1998), 217--237.


\bibitem{BST} \textsc{V.~I.~Burenkov, A.~Senouci, T.~V.~Tararykova}, Hardy-type inequality for $0 < p < 1$ and hypodecreasing functions, Eurasian mathematical journal, \textbf{1}:3 (2010), 27 -- 42.

\bibitem{Bur} \textsc{V.~I.~Burenkov}, Function spaces. Main integral inequalities related to $L_p$-spaces. Peoples' Friendship University, Moscow, 1989 (in Russian).

\bibitem{Bur1}  \textsc{V.~I.~Burenkov}, On the exact constant in the Hardy inequality with $0 < p < 1$ for monotone. Proc. Steklov Inst. Math., \textbf{194}: 4 (1993), 59 -- 63.

\bibitem{Dav}\textsc{E.~B.~Davies}, The Hardy constant, Quart. J. Math. Oxford Ser. (2), \textbf{46}:4 (1995), 417-431.

\bibitem{Dav1}\textsc{E.~B.~Davies}, A review of Hardy inequalities, The Maz'ya anniversary collection, v. 2 (Rostock, 1998), Oper. Theory Adv. Appl., 110, BirkhЁauser, Basel, 1999, 55-67.

\bibitem{EL} \textsc{W.~D.~Evans, R.~T.~Lewis}, Hardy and Rellich inequalities with remainders, J. Math. Inequal. \textbf{1}:4 (2007), 473--490.

\bibitem{F} \textsc{J.~L.~Fern\'{a}ndez}, Domains with Strong Barrier, Revista Matematica Iberoamericana,  \textbf{5}:2 (1989), 47--65.


\bibitem{FR} \textsc{J.~L.~Fern\'{a}ndez, J.~M.~Rodr\'{i}guez}, The exponent of convergence of Riemann surfaces, bass Riemann surfaces, Ann. Acad. Sci. Fenn. Series A. I. Mathematica,  \textbf{15} (1990), 165-183.

\bibitem{FMT} \textsc{S.~Filippas, V.~G.~Maz'ya, A.~Tertikas}, On a question of Brezis and Marcus, Calc. Var. Partial Differential Equations, \textbf{25}:4 (2006), 491--501.


\bibitem{HoHoL} \textsc{M.~Hoffmann-Ostenhof, T.~Hoffmann-Ostenhof, A.~Laptev}, A geometrical version of Hardy's inequality, J. Funct. Anal. \textbf{189}:2 (2002), 539--548.

\bibitem{HLP} \textsc{G.~H.~Hardy, J.~E.~Littlewood, G.~Polya}, Inequalities, Cambridge Univ. Press, Cambridge, 1973.

\bibitem{MatSob} \textsc{T.~Matskewich, P.~E.~Sobolevskii}, The best possible constant in generalized Hardy's inequality for convex domains in $\mathbb{R}^n$, Nonlinear Anal., \textbf{28}:9 (1997), 1601--1610.


\bibitem{MMP} \textsc{M.~Marcus, V.~J.~Mitzel, Y.~Pinchover}, On the best constant for Hardy's inequality in $R^n$, Trans. Amer. Math. Soc. \textbf{350}:8 (1998), 3237--3255.


\bibitem{Maz} \textsc{V.~G.~Maz'ya}, \emph{Sobolev spaces}, Springer (1985)

\bibitem{Nas_Sib_2019} \textsc{R.~G.~Nasibullin}, Brezis-Marcus type inequalities with Lamb constant, Sib.$\grave{E}$lektron. Mat. Izv., \textbf{16} (2019), 449--464.

\bibitem{Nas_Ufa_2017} \textsc{R.~G.~Nasibullin}, Sharp Hardy type inequalities with weights depending on Bessel function, Ufa Math. J., \textbf{9}:1 (2017), 89--97.

\bibitem{Nas_LJM_2019} \textsc{R.~G.~Nasibullin}, Multidimensional Hardy Type Inequalities with Remainders, Lobachevskii J. Math. \textbf{40}:9 (2019), 1383--1396.

\bibitem{Nas_LJM_2016} \textsc{R.~G.~Nasibullin}, Hardy type inequalities with weights dependent on the bessel functions Lobachevskii Journal of Mathematics. \textbf{37}:3 (2016),  274--283.

\bibitem{Nas_MS_2019} \textsc{R.~G.~Nasibullin}, A geometrical version of Hardy-Rellich type inequalities, Mathematica Slovaca. \textbf{69}:4 (2019), 785-800

\bibitem{PS} \textsc{G.~Psaradakis},  $L_1$ Hardy inequalities with weights, J. Geom. Anal. \textbf{23}:4, (2013) 1703--1728.

\bibitem{Shum}  \textsc{D.~T.~Shum}, On a class of new inequalities, Trans. Amer. Math. Soc. \textbf{204} (1975), 299--341.

\bibitem{Sobolev}  \textsc{S.~L.~Sobolev}, \emph{Selected problems in the theory of function spaces and generalized functions}, Nauka, Moscow, 1989. 255 pp. ISBN: 5-02-000052-3 [Russian]

\bibitem{TZ} \textsc{A.~Tertikas, N.~.B.~Zographopoulos}, Best constants in the Hardy-Rellich inequalities and related improvements, Advances in Mathematics, \textbf{209}:2 2007, 407--459.

\bibitem{Tidblom} \textsc{J.~Tidblom},  A geometrical version of Hardy's inequality for $ W^{1,p}_0(\Omega)$, Proc. Amer. Math. Soc. \textbf{132} (2004), 2265--2271.

\bibitem{Tidblom1} \textsc{J.~Tidblom},  A Hardy inequality in the half-space, Journal of Functional Analysis,  \textbf{221}:2 (2005), 482--495.

\bibitem{Watson} \textsc{G.~N~Watson}, \emph{A threatise on the theory of the Bessel Functions}, Cambridge: Cambridge Univ. Press, 1966.



\end{thebibliography}

\end{document}